\documentclass{article}

\newenvironment{proof}{\begin{trivlist} \item[]
{\bf Proof.}}{\nolinebreak
\hfill \rule{2mm}{2mm} \end{trivlist}}

\newcounter{ctr}

\newcounter{ctr1}

\newcounter{ctr2}

\newcounter{ctr3}

\newtheorem{definition}{Definition}[section]    
\newtheorem{theorem}[definition]{Theorem}
\newtheorem{lemma}[definition]{Lemma}
\newtheorem{corollary}[definition]{Corollary}
\newtheorem{proposition}[definition]{Proposition}

\newcommand{\NN}{\hbox{{\sf I}\kern-.15em\hbox{\sf N}}}
\newcommand{\ZZ}{\hbox{{\sf Z}\kern-.77em\hbox{\sf Z}\kern.15em}}
\newcommand{\QQ}{\hbox{{\sf I}\kern-.4em\hbox{\sf Q}}}
\newcommand{\RR}{\hbox{{\sf I}\kern-.15em\hbox{\sf R}}}
\newcommand{\CC}{\hbox{{\sf I}\kern-.4em\hbox{\sf C}}}

\newcommand{\ud}{\, {\rm d} \kern-.015em }

\newcommand{\mod}[1]{\left| \kern.05em #1 \kern.05em \right|}
\newcommand{\norm}[1]{\left\| \kern.05em #1 \kern.05em \right\|}
\newcommand{\inner}[1]{\left\langle \kern.05em #1 \kern.05em \right\rangle }

\newcommand{\pick}[2]{\renewcommand{\arraystretch}{0.6}
\left( \kern-.4em \begin{array}{c} #1 \\ #2 \end{array} \kern-.4em \right) }

\usepackage{amssymb}
\usepackage{young}
\usepackage[vcentermath]{youngtab}
\begin{document}

\begin{center} \Large \textbf{The hook fusion procedure for Hecke algebras}\\
\quad \\
\large James Grime \\ \normalsize \emph{Department of Mathematics,
University of York, York, YO10 5DD, UK} \\
\emph{jrg112@york.ac.uk}
\end{center}

\textbf{Abstract}\\ We derive a new expression for the $q$-analogue of the Young symmetrizer which generate irreducible representations of the Hecke algebra. We obtain this new expression using Cherednik's fusion
procedure. However, instead of splitting Young diagrams into their
rows or columns, we consider their principal hooks. This minimises
the number of auxiliary parameters needed in the fusion procedure.

\large

\section{Introduction}

In this article we present a new expression for certain elements in the Hecke algbra, related to the $q$-analogue of the Young symmetrizer, which generate its irreducible representations. We will obtain this new expression using Cherednik's fusion
procedure. This method originates from the work of Jucys
\cite{J}, and has already been used by Nazarov and Tarasov
\cite{N1, N2, NT}. However our approach differs by minimising the
number of auxiliary parameters needed in the fusion procedure.
This is done by considering hooks of Young diagrams, rather than
their rows or columns as in \cite{N1, N2, NT}.


Let $H_n$ be the finite 
dimensional Hecke algebra over the
field $\mathbb{C}(q)$ of rational functions in $q$,
with the generators $T_1, \ldots , T_{n-1}$ and the relations 
\begin{equation}\label{Hecke1}
(T_i-q)(T_i+q^{-1})=0;
\end{equation}
\begin{equation}\label{Hecke2}
T_i T_{i+1} T_i=T_{i+1} T_i T_{i+1};
\end{equation}
\begin{equation}\label{Hecke3}
T_i T_j=T_j T_i,
\quad
j\neq i, i+1
\end{equation}
for all possible indices $i$ and $j$.

The generators $T_1, \ldots , T_{n-1}$ are invertible since 
\begin{equation}\label{Tinverse}
T^{-1}_i=T_i-q+q^{-1}
\end{equation}
due to (\ref{Hecke1}). 

For any index $i=1, \ldots , n-1$ let $\sigma_i=(i, i+1)$ be 
the adjacent transposition in the symmetric group $S_n$. 
Take any element $\sigma \in S_n$
and choose a reduced decomposition $\sigma=\sigma_{i_1}\ldots \sigma_{i_l}$.
As usual put $T_\sigma=T_{i_1}\ldots T_{i_l}$,
this element of the algebra $H_n$ does not depend on the
choice of reduced decomposition of $\sigma$ due to (\ref{Hecke2}) and
(\ref{Hecke3}).  The element of maximal length in
$S_n$ will be denoted by $\sigma_0$. We will write $T_0$ instead of 
$T_{\sigma_0}$ for  short. The elements $T_\sigma$ form a basis of $H_n$ 
as a vector space over the field $\mathbb{C}(q)$. We will also use the 
basis in $H_n$  formed by the elements~$T_\sigma^{-1}$.


A \emph{partition} of $n$ is a sequence of weakly decreasing integers $\lambda_1 \geqslant \lambda_2 \geqslant \cdots \geqslant \lambda_k$ whose sum is equal to $n$. The \emph{Young diagram} of a
partition $\lambda$ is the set of boxes $(i, j) \in \mathbb{Z}^2$
such that $1 \leqslant j \leqslant \lambda_i$. In drawing such
diagrams we let the first coordinate $i$ increase as one goes
downwards, and the second coordinate $j$ increase from left to
right. For example the partition $\lambda = (3,3,2)$ gives the
diagram \[ \yng(3,3,2)\]

If $(i,j)$ is a box in the diagram of $\lambda$, then the
$(i,j)$-\emph{hook} is the set of boxes in $\lambda$
\[  \{ (i, j') : j' \geqslant j \} \cup \{ (i', j)
: i' \geqslant i \}, \] We call the $(i,i)$-hook the
\emph{$i^{\mbox{\small th}}$ principal hook}.

A \emph{standard tableau}, $\Lambda$, is a filling of the diagram
$\lambda$ in which the entries are the numbers 1 to $n$, each
occurring once. If the box $(i,j)$ contains $a$ we define the \emph{content} of the box to be $c_a(\Lambda) = j-i$.


The $\mathbb{C}(q)$-algebra $H_n$ is semisimple; see \cite[Section 4]{GU}
for a short proof of this well known fact. 
The simple ideals of $H_n$ are
labeled by partitions $\lambda$ of $n$, like the equivalence classes of 
irreducible representations of the symmetric group $S_n$.

In this article, for any standard tableau $\Lambda$ of shape $\lambda$ we 
will construct
a certain non-zero element $F_\Lambda\in H_n$. Under left multiplication by
the elements of $H_n$, the left ideal $H_nF_\Lambda\subset H_n$ is an 
irreducible $H_n$-module. The $H_n$-modules $V_\lambda$
for different partitions $\lambda$ are pairwise non-equivalent; see
Corollary \ref{C3.6}. At $q=1$, the algebra $H_n(q)$ specializes
to the group ring $\mathbb{C}S_n$, where $T_\sigma$ becomes the permutation $\sigma \in S_n$.  The $H_n(q)$-module
$V_\lambda$ then specializes to the irreducible representation of 
$S_n$, coresponding to the partition $\lambda$, \cite{Y1}.

Our construction of $V_\lambda$ employs a certain limiting process called 
the \emph{fusion procedure}. The idea of this construction goes 
back to \cite[Section 3]{C2} were no proofs were given however.
The element $F_\Lambda$ is related to the $q$-analogue
of the Young symmetrizer in the group ring $\mathbb{C} S_n$
constructed in \cite{Gy}.


For each $i=1, \ldots , n-1$ introduce the $H_n$-valued rational function in 
two variables $a \neq 0$, $a \neq b \in\mathbb{C}(q)$
\begin{equation}\label{q-smallf}
F_i(a, b)=T_i+\frac{q-q^{-1}}{a^{-1}b - 1}.
\end{equation}

Now introduce $n$ variables $z_1, \ldots , z_n\in\mathbb{C}(q)$.
Equip the set of all pairs $(i, j)$
where $1\leqslant i<j\leqslant n$, with the following ordering.
The pair $(i, j)$ precedes another pair
$(i',j')$ 
if $j<j'$, or if $j=j'$ but $i<i'$. Call this the \emph{reverse-lexicographic} ordering.
Take the ordered product
\begin{equation}\label{q-bigf}
F_\Lambda(z_1, \ldots , z_n) = \prod_{(i,j)}^{\rightarrow}\ F_{j-i}
\bigl( q^{2c_i(\Lambda)}z_i, q^{2c_j(\Lambda)}z_j\bigr)
\end{equation}
over this set. Consider the product (\ref{q-bigf})
as a rational function taking values in $H_n$,
of the variables $z_1, \ldots , z_n$.
If $i$ and $j$ sit on
the same diagonal in the tableau $\Lambda$, then $F_{j-i}
\bigl( q^{2c_i(\Lambda)}z_i, q^{2c_j(\Lambda)}z_j\bigr)$ has a pole at $z_i = z_j \neq 0$.

Let $\mathcal{R}_\Lambda$ be the vector subspace in $\mathbb{C}(q)^n$
consisting of all tuples $(z_1, \dots , z_n)$ such that $z_i =
z_j$ whenever the numbers $i$ and $j$ appear in the same row of
the tableau $\Lambda$.

As a direct calculation using (\ref{Hecke1}) and (\ref{Hecke2}) shows,
these functions satisfy 
\begin{equation}\label{q-triple}
F_i(a, b)F_{i+1}(a, c)F_i(b, c)=
F_{i+1}(b, c)F_i(a, c)F_{i+1}(a, b).
\end{equation}
Due to (\ref{Hecke3}) these rational functions also satisfy the relations
\begin{equation}\label{q-commute}
F_i(a, b) F_j(c, d)=F_j(c, d) F_i(a, b);
\qquad
j\neq i, i+1.
\end{equation}

Using (\ref{q-triple}) and (\ref{q-commute}) we may reorder the
product $F_\Lambda(z_1, \dots , z_n)$ such that each singularity
is contained in an expression known to be regular at $z_1 = z_2 =
\dots = z_n \neq 0$, \cite{N2}. It is by this method that it was shown
that the restriction of the rational function $F_\Lambda(z_1,
\dots ,z_n)$ to the subspace $\mathcal{R}_\Lambda$ is regular at
$z_1 = z_2 = \dots = z_n \neq 0$. Furthermore the following theorem was proved;

\begin{theorem} \label{q-MNtheorem} Restriction to $\mathcal{R}_\Lambda$ of the
rational function $F_\Lambda(z_1, \dots , z_n)$ is regular at $z_1
= z_2 = \dots = z_n \neq 0$ and has value $F_\Lambda \in H_n$. The left ideal generated by this element is irreducible, and the $H_n$-modules for different partitions $\lambda$ are pairwise non-equivalent.
\end{theorem}

Similarly, we may form another expression for $F_\Lambda$ by
considering the subspace in $\mathbb{C}(q)^n$ consisting of all
tuples $(z_1, \dots , z_n)$ such that $z_i = z_j$ whenever the
numbers $i$ and $j$ appear in the same column of the tableau
$\Lambda$ \cite{N1}.

In this article we present a new expression for the element $F_\Lambda$ which minimises the number of auxiliary parameters
needed in the fusion procedure. We do this by considering hooks of
standard tableaux rather than their rows or columns.

Let $\mathcal{H}_\Lambda$ be the vector subspace in $\mathbb{C}(q)^n$
consisting of all tuples $(z_1, \dots , z_n)$ such that $z_i =
z_j$ whenever the numbers $i$ and $j$ appear in the same principal
hook of the tableau $\Lambda$. We will prove the following
theorem.

\begin{theorem}\label{q-fulltheorem} Restriction to $\mathcal{H}_\Lambda$ of the
rational function $F_\Lambda(z_1, \dots , z_n)$ is regular at $z_1
= z_2 = \dots = z_n \neq 0$ and has value $F_\Lambda \in H_n$. The left ideal generated by this element is irreducible, and the $H_n$-modules for different partitions $\lambda$ are pairwise non-equivalent. The element is the same as the element in Theorem \ref{q-MNtheorem}.
\end{theorem}

In particular, this hook fusion procedure can be used to form
irreducible representations of $H_n$ corresponding to Young
diagrams of hook shape using only one auxiliary parameter, $z$. By
taking this parameter to be 1 we find that no parameters are
needed for diagrams of hook shape. Therefore if $\nu$ is a
partition of hook shape, and $N$ a standard tableau of shape $\nu$, we have
\begin{equation}\label{q-bigfhook} F_N = F_N(z) =
\prod_{(p,q)}^\rightarrow F_{j-i}(q^{2c_i(\Lambda)}, q^{c_j(\Lambda)}), \end{equation} with
the pairs $(i, j)$ in the product ordered reverse-lexicographically.


To motivate the study of modules corresponding to partitions of
hook shape first let us consider  the \emph{Jacobi-Trudi
identities} \cite[Chapter I3]{MD}. There is an isomorphism from the ring of symmetric functions to the Grothendieck ring of representations of the Hecke algebra. Therefore we can think of the following
identities as dual to the Jacobi-Trudi identities.

If $\lambda = (\lambda_1, \dots , \lambda_k)$ such that $n =
\lambda_1 + \dots + \lambda_k$ then we have the following
decomposition of the induced representation of the tensor product
of modules corresponding to the rows of $\lambda$;
\[ \textrm{Ind}_{H_{\lambda_1} \times H_{\lambda_2} \times \cdots
\times H_{\lambda_k}}^{H_n } V_{(\lambda_1)} \otimes
V_{(\lambda_2)} \otimes \cdots \otimes V_{(\lambda_k)} \cong
\bigoplus_{\mu} (V_\mu)^{\oplus K_{\mu \lambda}},\] where the sum
is over all partitions of $n$. Note that $V_{(\lambda_i)}$ is the
trivial representation of $H_{\lambda_i}$, that sends generators $T_i$ to $q$. The coefficients
$K_{\mu \lambda}$ are non-negative integers known as \emph{Kostka
numbers}, \cite{MD}. Importantly, we have $K_{\lambda \lambda} =
1.$

On the subspace $\mathcal{R}_\Lambda$, if $z_i / z_j \notin
q^\mathbb{Z}$ when $i$ and $j$ are in different rows of $\Lambda$
then the above induced module may be realised as the left ideal in
$H_n$ generated by $F_\Lambda(z_1, \dots, z_n)$.\\
The irreducible representation $V_\lambda$ appears in the
decomposition of this induced module with coefficient 1, and is
the ideal of $H_n$ generated by $F_\Lambda(z_1, \dots ,
z_n)$ when $z_1 = z_2 = \dots = z_n \neq 0$. The fusion procedure of
Theorem \ref{q-MNtheorem} provides a way of singling out this
irreducible component.

Similarly we have the equivalent identity for columns, \[
\textrm{Ind}_{H_{\lambda'_1} \times H_{\lambda'_2} \times \cdots
\times H_{\lambda'_l}}^{H_n } V_{(1^{\lambda'_1})} \otimes
V_{(1^{\lambda'_2})} \otimes \cdots \otimes V_{(1^{\lambda'_l})}
\cong \bigoplus_{\mu} (V_{\mu})^{\oplus K_{\mu' \lambda'}},\]
where $l$ is the number of columns of $\lambda$. In this case
$V_{(1^{\lambda'_i})}$ is the alternating representation of
$H_{\lambda'_i}$. This induced module is isomorphic to the left
ideal of $H_n$ generated by $F_{\Lambda}(z_1, \dots,
z_n)$ considered on the subspace $\mathcal{R}_{\Lambda'}$, with
$z_i / z_j \notin q^\mathbb{Z}$ when $i, j$ are in different columns
of $\Lambda$. Again the irreducible representation $V_{\lambda}$
appears in the decomposition of this induced module with
coefficient 1, and is the ideal of $H_n$ generated by
$F_{\Lambda}(z_1, \dots, z_n)$ when $z_1 = z_2 = \dots = z_n \neq 0$.

There is another expression known as the \emph{Giambelli identity}
\cite{Gi}. Unlike the Jacobi-Trudi identities, this identity
involves splitting $\lambda$ into its principal hooks, rather than
its rows or columns. A combinatorial proof of the Giambelli
identity can be found in \cite{ER}.

Divide a Young diagram $\lambda$ into boxes with positive and
non-positive content. We may illustrate this on the Young diagram
by drawing 'steps' above the main diagonal. Denote the boxes above
the steps by $\alpha(\lambda)$ and the rest by $\beta(\lambda)$.
For example, the following figure illustrates $\lambda$,
$\alpha(\lambda)$ and $\beta(\lambda)$ for $\lambda = (3,3,2)$.

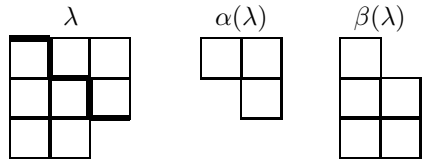
\begin{figure}[h]
\label{steps}
\begin{center}
\begin{picture}(150, 60)
\put(0,30){\framebox(15,15)[r]{  }}
\put(15,30){\framebox(15,15)[r]{ }}
\put(30,30){\framebox(15,15)[r]{  }}

\put(0,15){\framebox(15,15)[r]{  }}
\put(15,15){\framebox(15,15)[r]{  }}
\put(30,15){\framebox(15,15)[r]{  }}

\put(0,0){\framebox(15,15)[r]{  }}
\put(15,0){\framebox(15,15)[r]{}}

\put(72.5,30){\framebox(15,15)[r]{  }}
\put(87.5,30){\framebox(15,15)[r]{ }}
\put(125,30){\framebox(15,15)[r]{ }}

\put(87.5,15){\framebox(15,15)[r]{  }}
\put(125,15){\framebox(15,15)[r]{  }}
\put(140,15){\framebox(15,15)[r]{  }}

\put(125,0){\framebox(15,15)[r]{ }}
\put(140,0){\framebox(15,15)[r]{}}

\put(20,50){$\lambda$} \put(77.5,50){$\alpha(\lambda)$}
\put(130,50){$\beta(\lambda)$}

\linethickness{2pt} \put(0,45){\line(1,0){15}}
\put(15,30){\line(1,0){15}} \put(30,15){\line(1,0){15}}

\put(15,30){\line(0,1){15}} \put(30,15){\line(0,1){15}}

\end{picture}
\end{center}
 \caption{The Young diagram $(3,3,2)$ divided into boxes with positive content and non-positive content}
 \end{figure}

If we denote the rows of $\alpha(\lambda)$ by $\alpha_1 > \alpha_2
> \dots > \alpha_d > 0$ and the columns of $\beta(\lambda)$
by $\beta_1 > \beta_2 > \dots > \beta_d > 0$, then we have the
following alternative notation for $\lambda$;
\[ \lambda = ( \alpha | \beta ), \]
where $\alpha = (\alpha_1, \dots , \alpha_d)$ and $\beta =
(\beta_1, \dots , \beta_d)$.

Here $d$ denotes the length of the side of the \emph{Durfee
square} of shape $\lambda$, which is the set of boxes
corresponding to the largest square that fits inside $\lambda$,
and is equal to the number of principal hooks in $\lambda$. In our
example $d=2$ and $\lambda = (2, 1 | 3, 2)$.

We may consider the following identity as a dual of the Giambelli
identity.
\[ \textrm{Ind}_{H_{h_1} \times H_{h_2} \times \cdots
\times H_{h_d}}^{H_n} V_{(\alpha_1 | \beta_1)} \otimes
V_{(\alpha_2 | \beta_2)} \otimes \cdots \otimes V_{(\alpha_d |
\beta_d)} \cong \bigoplus_{\mu} (V_{\mu})^{\oplus D_{\mu
\lambda}},\] where $h_i$ is the length of the $i^{\mbox{\small th}}$
principal hook, and the sum is over all partitions of $n$. This is
a decomposition of the induced representation of the tensor
product of modules of hook shape. Further these hooks are the
principal hooks of $\lambda$. The coefficients, $D_{\mu \lambda}$,
are non-negative integers, and in particular $D_{\lambda \lambda}
=1$.

On the subspace $\mathcal{H}_\Lambda$, if $z_i / z_j \notin
q^\mathbb{Z}$ when $i$ and $j$ are in different principal hooks of
$\Lambda$ then the above induced module may be realised as the
left ideal in
$H_n$ generated by $F_\Lambda(z_1, \dots, z_n)$.\\
The irreducible representation $V_\lambda$ appears in the
decomposition of this induced module with coefficient 1, and is
the ideal of $H_n$ generated by $F_\Lambda(z_1, \dots ,
z_n)$ when $z_1 = z_2 = \dots = z_n$.

Hence, in this way,  our hook fusion procedure relates to the
Giambelli identity in the same way that Cherednik's original
fusion procedure relates to the Jacobi-Trudi identity. Namely, it
provides a way of singling out the irreducible component
$V_\lambda$ from the above induced module.

The fusion procedure was originally developed in the study of
affine Hecke algebras, \cite{C1}. Our results may be regarded as an
application of the representation theory of these algebras,
\cite{OV}. Descriptions of the fusion procedure for the Symmetric group may be found in \cite{NT} and \cite{GP}. The hook fusion procedure for the Symmetric group was considered in \cite{Gr}.


Acknowledgements and thanks go to Maxim Nazarov for his
supervision, and for introducing me to this subject. I would also
like to thank EPSRC for funding my research.

\section{Fusion Procedure for a Young Diagram}


We fill a diagram $\lambda$ by hooks to form a tableau $\Lambda^\circ$
in the following way: For the first principal hook we fill the
column with entries $1$, $2$, \dots , $r$ and then fill the row
with entries $r+1$, $r+2$, \dots , $s$. We then fill the column of
the second principal hook with $s+1$, $s+2$, \dots , $t$ and fill
the row with $t+1$, $t+2$, \dots , $x$. Continuing in this way we
form the hook tableau.

{\addtocounter{definition}{1} \bf Example \thedefinition .} On the
left is the hook tableau of the diagram $\lambda = (3,3,2)$, and
on the right the same diagram with the content of each box.

\begin{normalsize}

\begin{center}
\begin{picture}(50,50)
\put(0,30){\framebox(15,15)[r]{ 1 }}
\put(15,30){\framebox(15,15)[r]{ 4 }}
\put(30,30){\framebox(15,15)[r]{ 5 }}
\put(0,15){\framebox(15,15)[r]{ 2 }}
\put(15,15){\framebox(15,15)[r]{ 6 }}
\put(30,15){\framebox(15,15)[r]{ 8 }}
\put(0,0){\framebox(15,15)[r]{ 3 }}
\put(15,0){\framebox(15,15)[r]{ 7 }}
\end{picture}
\qquad \qquad \qquad
\begin{picture}(50,50)
\put(0,30){\framebox(15,15)[r]{ 0 }}
\put(15,30){\framebox(15,15)[r]{ 1 }}
\put(30,30){\framebox(15,15)[r]{ 2 }}
\put(0,15){\framebox(15,15)[r]{ -1 }}
\put(15,15){\framebox(15,15)[r]{ 0 }}
\put(30,15){\framebox(15,15)[r]{ 1 }}
\put(0,0){\framebox(15,15)[r]{ -2 }}
\put(15,0){\framebox(15,15)[r]{ -1 }}
\end{picture}
\end{center}
\end{normalsize}
Therefore the sequence $(c_1(\Lambda^\circ), c_2(\Lambda^\circ), \dots , c_8(\Lambda^\circ))$ is given by $(0,
-1, -2, 1, 2, 0 , -1, 1)$. {\nolinebreak \hfill \rule{2mm}{2mm}

Consider (\ref{q-bigf}) as a rational function of the variables
$z_1, \dots , z_n$ with values in $H_n$. Using the substitution \begin{equation}\label{substitution} w_i = q^{c_i(\Lambda)}z_i, \end{equation} the factor
$F_{i}(w_a, w_b)$ has a pole at $z_a = z_b$ if and
only if the numbers $a$ and $b$ stand on the same diagonal of a
tableau $\Lambda$. We then call the pair $(a, b)$ a
\emph{singularity}. And we call the corresponding term $F_{i}(w_a,w_b)$ a \emph{singularity term}, or
simply a singularity.

Let $a$ and $b$ be in the same principal hook of $\Lambda$. If $a$ and $b$
are next to one another in the column of the hook then, on
$\mathcal{H}_\lambda$, $F_{i}(w_a,w_b) = T_{i} - q$. Since \begin{equation}\label{P-}(T_i - q)^2 = -(q+q^{-1})(T_i-q)\end{equation}
then $\frac{-1}{q+q^{-1}} F_{i}(w_a,w_b)$ is an
idempotent. Denote this idempotent $P^-_i$.

Similarly, if $a$ and $b$ are next to one another in
the same row of the hook then $F_{i}(w_a,w_b)= T_{i} +
q^{-1}$. And since \begin{equation}\label{P+}(T_i + q^{-1})^2 = (q+q^{-1})(T_i+q^{-1})\end{equation} then
$\frac{1}{q+q^{-1}} F_{i}(w_a,w_b)$ is an idempotent. Denote this idempotent $P^+_i$.\\

We also have
\begin{equation}\label{q-inverse} F_i(a, b) F_i(b, a)=1-\frac{(q-q^{-1})^2 ab}{(a-b)^2}.
\end{equation}
Therefore, if the contents $c_a(\Lambda)$ and $c_b(\Lambda)$ differ by a number
greater than one, then the factor $F_{i}(w_a,w_b)$
is invertible in $H_n$ when $z_a = z_b \neq 0$ for all values of $q$.

The presence of singularity terms in the product $F_\Lambda(z_1,
\dots , z_n)$ mean this product may or may not be regular on the
vector subspace of $\mathcal{H}_\lambda$ consisting of all tuples
$(z_1, \dots , z_n)$ such that $z_1 = z_2 = \dots = z_n \neq 0$. Using
the following lemma, we will be able to show that $F_\Lambda(z_1,
\dots , z_n)$ is indeed regular on this subspace.

\begin{lemma}\label{q-regular} Restriction of the rational function $F_i(a, b)F_{i+1}(a, c)F_i(b, c)$ to the set of
$(a, b, c)$ such that $a=q^{\pm 2}b$, 
is regular at $a = c \neq 0$. \end{lemma}
\begin{proof} Let us expand the product at the left hand side of (\ref{q-triple}) in the
factor
$F_{i+1}(a, c)$. By the definition (\ref{q-smallf}) we will get the sum
\[
F_i(a, b)T_{i+1}F_i(b, c)+
\frac{q-q^{-1}}{a^{-1}c-1}F_i(a, b)F_i(b, c).
\]
Here the restriction to $a=q^{\pm 2}b$ of the first summand is
evidently regular at $a=c$. After the substitution 
$b=q^{\mp2}a$, the second summand takes the form
\[
\frac{q-q^{-1}}{a^{-1}c-1}
\bigl( T_i\mp q^{\pm1}\bigr)
\biggl(
T_i+\frac{q-q^{-1}}{q^{\pm2}a^{-1} c-1}
\biggr)
=
\frac{q-q^{-1}}{a^{-1} c-q^{\mp2}}
( q^{\pm1}\mp T_i).
\]
The rational function of $a, c$ at the right hand side
of the last displayed equality is also evidently regular at $a=c$. \end{proof}

In particular, if the middle term on the left hand side of (\ref{q-triple}) is a singularity
and the other two terms are an appropriate idempotent and
\emph{triple term}, then the three term product, or \emph{triple} is regular at $z_1 = z_2 =
\dots = z_n \neq 0$. we may now prove the first statement of Theorem
\ref{q-fulltheorem}.

\begin{proposition}\label{q-jimtheorem1} The restriction of the rational function
$F_\Lambda (z_1, \dots , z_n)$ to the subspace
$\mathcal{H}_\lambda$ is regular at $z_1 = z_2 = \dots = z_n \neq 0$.
\end{proposition}

\begin{proof}
Consider any standard tableau $\Lambda'$ obtained from the tableau $\Lambda$
by an adjacent transposition of its entries, say by $\sigma_k\in S_n$.
Using the relations (\ref{q-triple}) and (\ref{q-commute}), we derive
the equality of rational functions in the variables $z_1, \ldots , z_n$
\[
F_\Lambda(z_1, \ldots , z_n)F_{n-k}
\bigl( 
q^{2c_{k+1}(\Lambda)}z_{k+1}, q^{2c_k(\Lambda)}z_k
\bigr)
=
\]
\begin{equation}\label{q-jimtheorem1equation}
F_k
\bigl( 
q^{2c_k(\Lambda)}z_k, q^{2c_{k+1}(\Lambda)}z_{k+1}
\bigr)
F_{\Lambda^{\prime}}(z'_1, \ldots ,z'_n),
\end{equation}
where the sequence of variables $(z'_1, \ldots ,z'_n)$\ is obtained from
the sequence $(z_1, \ldots , z_n)$ by exchanging the terms $z_k$ and 
$z_{k+1}$. Observe that
\[
(z'_1, \ldots ,z'_n)\in\mathcal{H}_{\Lambda'}
\quad\Leftrightarrow\quad
(z_1, \ldots , z_n)\in\mathcal{H}_\Lambda.
\]
Also observe that here $| c_k(\Lambda)-c_{k+1}(\Lambda)|\geqslant2$
because the tableaux $\Lambda$ and $\Lambda'$ are standard.
Therefore the functions 
\[
F_k
\bigl( 
q^{2c_k(\Lambda)}z_k, q^{2c_{k+1}(\Lambda)}z_{k+1}
\bigr)
\ \quad\textrm{and}\ \quad
F_{n-k}
\bigl( 
q^{2c_{k+1}(\Lambda)}z_{k+1}, q^{2c_k(\Lambda)}z_k
\bigr)
\]
appearing in the equality (\ref{q-jimtheorem1equation}),
are regular at $z_k=z_{k+1} \neq 0$.
Moreover, their values at $z_k=z_{k+1} \neq 0$ are invertible
in the algebra $H_n$, see the relation (\ref{q-inverse}). 
Due to these two observations, the equality (\ref{q-jimtheorem1equation})
shows that Proposition \ref{q-jimtheorem1} is equivalent to its counterpart for
the tableau $\Lambda'$ instead of $\Lambda$.

Let us take the hook tableau $\Lambda^\circ$ of shape $\lambda$.  
There is a chain $\Lambda,\Lambda', \ldots ,\Lambda^\circ$ of standard tableaux
of the same shape $\lambda$, such that each subsequent tableau in the 
chain is
obtained from the previous one by an adjacent transposition of the 
entries.
Due to the above argument, it now suffices to prove Proposition \ref{q-jimtheorem1} 
only in the case $\Lambda=\Lambda^\circ$.

We will prove the statement by reordering the factors of the
product $F_{\Lambda^\circ} (z_1, \dots , z_n)$, using relations
(\ref{q-triple}) and (\ref{q-commute}), in such a way that each
singularity is part of a triple which is regular at $z_1 = z_2 =
\dots = z_n \neq 0$, and hence the whole of $F_{\Lambda^\circ} (z_1, \dots ,
z_n)$ will be manifestly regular.

Define $g(a,b)$ to be the following; \[ g(a,b) = \left\{
\begin{array}{ccc}
  F_{b-a} (w_a, w_b) & \textrm{if} & a<b \\
  1 & \textrm{if} & a \geqslant b
\end{array} \right. \]
where $w_i$ is the substitution (\ref{substitution}).

Now, let us divide the diagram $\lambda$ into two parts,
consisting of those boxes with positive contents and those with
non-positive contents as in Figure \ref{steps}. Consider the
entries of the $i^{\mbox{\small th}}$ column of the hook tableau
$\Lambda^\circ$ of shape $\lambda$ that lie below the steps. If $u_1,
u_2, \dots , u_k$ are the entries of the $i^{\mbox{\small th}}$ column
below the steps, we define
\begin{equation}\label{q-cproduct} C_i = \prod_{j=1}^n g(u_1 , j)
g(u_2 , j) \cdots g(u_k , j).
\end{equation} Now consider the entries of the $i^{\mbox{\small th}}$ row of $\Lambda^\circ$ that lie
above the steps. If $v_1, v_2, \dots , v_l$ are the entries of the
$i^{\mbox{\small th}}$ row above the steps, we define
\begin{equation}\label{q-rproduct} R_i = \prod_{j=1}^n g(v_1 , j)
g(v_2 , j) \cdots g(v_l , j).
\end{equation}

Our choice of the hook tableau was such that the following is
true; if $d$ is the number of principal hooks of $\lambda$ then by
relations (\ref{q-triple}) and (\ref{q-commute}) we may reorder the
factors of $F_{\Lambda^\circ} (z_1, \dots , z_n)$ such that \[ F_{\Lambda^\circ}
(z_1, \dots , z_n) = \prod_{i=1}^d C_i R_i . \]

Now, each singularity $(a,b)$ has its corresponding term $F_{b-a}
(w_a,w_b)$ contain in some product $C_i$ or $R_i$.
This singularity term will be on the immediate left of the
 term $F_{b-a-1} (w_{a+1}, w_b)$. Also, this ordering has been chosen such
that the product of factors to the left of any such singularity in
$C_i$ or $R_i$ is divisible on the right by $F_{b-a-1}(w_a,w_{b+1})$.\\
Therefore we can replace the pair $F_{b-a}
(w_a,w_b)F_{b-a-1} (w_{a+1}, w_b)$ in $C_i$ by the triple
\[ P_{b-a-1}^- F_{b-a}
(w_a,w_b)F_{b-a-1} (w_{a+1}, w_b), \] where
$P_{b-a-1}^- = \frac{-1}{q+q^{-1}} F_{b-a-1}(w_a,w_{a+1})$ is the
idempotent (\ref{P-}). Divisibility on the right by $F_{b-a-1}(w_a,w_{a+1})$ means the addition of the idempotent has no
effect on the value of the product $C_i$. Similarly, in the product $R_i$ we can replace the pair by \[ P_{b-a-1}^+ F_{b-a}
(w_a,w_b)F_{b-a-1} (w_{a+1}, w_b). \] 
By Lemma \ref{q-regular}, the above triples are regular at $z_1 =
z_2 = \dots = z_n \neq 0$, and therefore, so too are the products $C_i$
and $R_i$, for all $1 \leqslant i \leqslant d$. Moreover, this
means $F_{\Lambda^\circ} (z_1, \dots , z_n)$ is regular at $z_1 = z_2 =
\dots = z_n \neq 0$.
\end{proof}

{\addtocounter{definition}{1} \bf Example \thedefinition .} As an
example consider the hook tableau of the Young diagram $\lambda =
(3,3,2)$, as shown in Example 2.1.

In the original reverse-lexicographic ordering the product $F_{\Lambda^{\circ}}(z_1,
\dots , z_n)$ is written as;
\[
\begin{array}{rl}
   F_{\Lambda^\circ} (z_1, \dots , z_n) = & F_1(w_1,w_2)F_2(w_1,w_3)F_1(w_2,w_3)F_3(w_1,w_4)F_2(w_2,w_4)F_1(w_3,w_4)\\&
   F_4(w_1,w_5)F_3(w_2,w_5)F_2(w_3,w_5)F_1(w_4,w_5)F_5(w_1,w_6)F_4(w_2,w_6)\\&F_3(w_3,w_6)F_2(w_4,w_6)F_1(w_5,w_6)F_6(w_1,w_7)F_5(w_2,w_7)F_4(w_3,w_7)\\&F_3(w_4,w_7)F_3(w_4,w_7)F_2(w_5,w_7)F_1(w_6,w_7)F_7(w_1,w_8)F_6(w_2,w_8)\\&F_5(w_3,w_8)F_4(w_4,w_8)F_3(w_5,w_8)F_2(w_6,w_8)F_1(w_7,w_8)\\
\end{array}
\]
we may now reorder this product into the form below using
relations (\ref{q-triple}) and (\ref{q-commute}) as described in the
above proposition. The terms in square brackets are the singularity terms
with their appropriate triple terms.
\[
\begin{array}{rl}
   F_{\Lambda^\circ} (z_1, \dots , z_n) = &
   F_1(w_1,w_2)F_2(w_1,w_3)F_1(w_2,w_3)F_3(w_1,w_4)F_2(w_2,w_4)F_1(w_3,w_4)\\&F_4(w_1,w_5)F_3(w_2,w_5)F_2(w_3,w_5)\left[F_5(w_1,w_6)F_4(w_2,w_6)\right]F_3(w_3,w_6)\\&F_6(w_1,w_7)\left[F_5(w_2,w_7)F_4(w_3,w_7)\right]F_7(w_1,w_8)F_6(w_2,w_8)F_5(w_3,w_8)\\&
   \cdot F_1(w_4,w_5)F_2(w_4,w_6)F_1(w_5,w_6)F_3(w_4,w_7)F_2(w_5,w_7)\\&\left[F_4(w_4,w_8)F_3(w_5,w_8)\right] \cdot F_1(w_6,w_7)F_2(w_6,w_8)F_1(w_7,w_8)\\
\end{array}
\]
We may now add the appropriate idempotents to these singularity-triple term pairs to form triples. Since each of these triples are regular at $z_1 = z_2 = \dots
= z_n$ then so too is the whole of $F_{\Lambda^\circ} (z_1, \dots , z_n)$.
{\nolinebreak \hfill \rule{2mm}{2mm}

\quad

Therefore, due to the above proposition an element $F_\Lambda \in
H_n$ can now be defined as the value of $F_\Lambda (z_1,
\dots , z_n)$ at $z_1 = z_2 = \dots = z_n \neq 0$. Note that for $n=1$ we have $F_\Lambda=1$. For any $n\geqslant1$, 
take the expansion of the element $F_\Lambda\in H_n$ in the basis of the 
elements $T_\sigma$ where $\sigma$ is ranging over $S_n$.

\begin{proposition}\label{T_0coeff} The coefficient in $F_\Lambda\in H_n$ of the element $T_0$ is $1$.
\end{proposition}
\begin{proof} Expand the product (\ref{q-bigf})
as a sum of the elements $T_\sigma$ with coefficients
from the field of rational functions of $z_1, \ldots , z_n$;
these functions take values in $\mathbb{C}(q)$. 
The decomposition in $S_n$
with ordering of the pairs $(i, j)$ as in (\ref{q-bigf})
\[
\sigma_0=\prod_{(i,j)}^{\longrightarrow}\ \sigma_{j-i}
\]
is reduced, hence the coefficient at $T_0=T_{\sigma_0}$ in the
expansion of (\ref{q-bigf}) is $1$. By the definition of $F_\Lambda$,
then the coefficient of $T_0$ in $F_\Lambda$  must be also $1$
\end{proof}

In particular this shows that $F_\Lambda \neq 0$ for any nonempty
diagram $\lambda$. Let us now denote by $\varphi$ the involutive
antiautomorphism of the algebra $H_n$ over the field $\mathbb{C}(q)$, defined by
$\varphi (T_i) = T_i$ for every $i \in 1, \dots , n-1$.

\begin{proposition}\label{q-varphi} The element $F_\Lambda T^{-1}_0$ is $\varphi$-invariant. \end{proposition}
\begin{proof} Any element of the algebra $H_n$ of the form
$F_i(a, b)$ is $\varphi$-invariant. 
Hence applying the antiautomorphism $\varphi$ to an element of 
$H_n$ the form (\ref{q-bigf}) just reverses the ordering of
the factors corresponding to the pairs $(i, j)$.
Using the relations (\ref{q-triple}) and (\ref{q-commute}),
we can rewrite the reversed product as
\[
\prod_{(i,j)}^{\longrightarrow}\ F_{ n-j+i}
\bigl( q^{2c_i(\Lambda)}z_i, q^{2c_j(\Lambda)}z_j\bigr)
\]
where the pairs $(i, j)$ are again ordered as in (\ref{q-bigf}).
But due to (\ref{Hecke2}) and (\ref{Hecke3}),
we also have the identity in the algebra $H_n$
\[
F_{ n-i}(a,b)T_0=T_0 F_i(a, b).
\]
This identity along with the equality 
$\varphi(T_0)=T_0$ implies that
any value of the function $F_\Lambda(z_1, \ldots , z_n)T^{-1}_0$ is 
$\varphi$-invariant. So is the element $F_\Lambda T^{-1}_0 \in H_n$
\end{proof}

\begin{proposition}\label{q-stripping} If $\lambda= (\alpha_1,
\alpha_2, \dots, \alpha_d | \beta_1, \beta_2, \dots , \beta_d)$
and $\mu = (\alpha_{k+1}, \alpha_{k+2}, \dots , \alpha_d |
\beta_{k+1}, \beta_{k+2}, \dots , \beta_d)$, then $F_{\Lambda^\circ} = P
\cdot F_{\mathrm{M}^\circ}$, for some element $P \in H_n$.
\end{proposition}
\begin{proof}
Here the shape $\mu$ is obtained by removing the first $k$
principal hooks of $\lambda$. Let $x$ be last entry in the $k^{\mbox{\small th}}$ row
of the hook tableau of shape $\lambda$. By the ordering described in
Proposition \ref{q-jimtheorem1},
\[ F_{\Lambda^\circ}(z_1, \dots ,
z_n) = \prod_{i=1}^k C_iR_i \cdot F_{\mathrm{M}^\circ}(z_{x+1}, \dots,
z_{n}),\]
where $C_i$ and $R_i$ are defined by (\ref{q-cproduct}) and (\ref{q-rproduct}).\\
Since all products $C_i$ and $R_i$ are regular at $z_1 = z_2 =
\dots = z_n \neq 0$, Proposition \ref{q-jimtheorem1} then gives us the
required statement. \end{proof}

In any given ordering of $F_\Lambda(z_1, \dots , z_n)$, we want a
singularity term to be placed next to an appropriate triple term
such that we may then form a regular triple. In that case we will
say these two terms are 'tied'. However, proving the divisibilities described in the next two propositions require some pairs
to be 'untied', in which case we must form a new ordering. This is
the content of the following proofs. Some explicit examples will
then given in Example 2.10 below.

\begin{proposition}\label{q-jimtheorem2} Suppose the numbers $u < v$ stand next to each
other in the same column of the hook tableau $\Lambda^\circ$ of shape
$\lambda$. First, let $s$ be the last entry in the row containing
$u$. If $c_v < 0$ then the element $F_{\Lambda^\circ} \in H_n$
is divisible on the left 
by $F_{u}(q^{2c_u(\Lambda^\circ)}, q^{2c_v(\Lambda^\circ)}) =
T_{u} - q$. If $c_v \geqslant 0$ then the element $F_{\Lambda^\circ} \in
H_n$ is divisible on the left by the product
\[ \prod_{i = u, \dots, s}^\leftarrow \left( \prod_{j= s+1, \dots,
v}^\rightarrow F_{i+j-s-1}(q^{2c_i(\Lambda^\circ)}, q^{2c_j(\Lambda^\circ)}) \right) \]
\end{proposition}

\begin{proof}
%
%
Let $\lambda$ and $\mu$ be as in Proposition \ref{q-stripping} with $\lambda$ a partition of $n$ and $\mu$ a partition of $m$. If $F_\mathrm{M}(z_1, \dots, z_m)$ is divisible on the left by $F_i(w_a,w_b)$ then $F_\mathrm{M}(z_{x+1}, \dots, z_n)$ is divisible on the left by $F_i(w_{a+x},w_{b+x})$. Then, by Proposition \ref{q-varphi}, $F_\mathrm{M}(z_{x+1}, \dots, z_n)$ is divisible on the right by $F_{m-i}(w_{a+x},w_{b+x})$. Therefore, by Proposition \ref{q-stripping}, $F_{\Lambda^\circ}(z_1, \dots, z_n)$ is divisible on the right by $F_{m-i}(w_{a+x},w_{b+x})$ and so $F_{\Lambda^\circ}(z_1, \dots, z_n)$ is divisible on the left by $F_{n-m+i}(w_{a+x}, w_{b+x}) = F_{i+x}(w_{a+x},w_{b+x})$. Hence we only need to prove the statement for $(u,v)$ such that $u$ is in the first row or first column of $\Lambda^\circ$.

Let $r$ be the last entry in the first column of $\Lambda^\circ$, $s$
the last entry in the first row of $\Lambda^\circ$, and $t$ the last
entry in the second column of $\Lambda^\circ$, as shown in Figure
\ref{jimtheorem2fig}.

\begin{figure}[h] \label{jimtheorem2fig}

\begin{normalsize}
\begin{center}
\begin{picture}(275,200)
\begin{small}
\put(0,175){\framebox(25,25)[c]{ 1 }}
\put(25,175){\framebox(25,25)[c]{ $r + 1$ }}
\put(50,175){\framebox(25,25)[c]{ $r+2$ }}
\put(0,150){\framebox(25,25)[c]{ 2 }}
\put(25,150){\framebox(25,25)[c]{ $s+1$ }}
\put(50,150){\framebox(25,25)[c]{ $t+1$ }}

\put(0,25){\line(0,1){125}} \put(25,25){\line(0,1){125}}
\put(0,0){\framebox(25,25)[c]{ $r$ }}

\put(75,200){\line(1,0){50}} \put(75,175){\line(1,0){50}}
\put(125,175){\framebox(25,25)[c]{ $u$ }}
\put(150,200){\line(1,0){100}} \put(150,175){\line(1,0){100}}
\put(250,175){\framebox(25,25)[c]{ $s$ }}

\put(50,100){\line(0,1){50}} \put(25,75){\framebox(25,25)[c]{ $t$
}}

\put(75,150){\line(1,0){50}} \put(125,150){\framebox(25,25)[c]{
$v$ }} \put(150,150){\line(1,0){50}} \put(200,150){\line(0,1){25}}

\linethickness{1.5pt} \put(0,200){\line(1,0){25}}
\put(25,175){\line(1,0){25}} \put(50,150){\line(1,0){25}}
\put(25,175){\line(0,1){25}} \put(50,150){\line(0,1){25}}

\put(125,150){\line(0,1){50}}\put(150,150){\line(0,1){50}}
\put(125,200){\line(1,0){25}}\put(125,150){\line(1,0){25}}

\put(95,185){$\dots$} \put(170,185){$\dots$}\put(220,185){$\dots$}
\put(95,160){$\dots$} \put(170,160){$\dots$}

\put(10,120){$\vdots$} \put(35,
120){$\vdots$}\put(10,80){$\vdots$} \put(10,40){$\vdots$}

\end{small}
\end{picture}
\end{center}
\end{normalsize}
 \caption{The first two principal hooks of the hook tableau $\Lambda^\circ$}
 \end{figure}

We now continue this proof by considering three cases and showing
the appropriate divisibility in each.

\emph{(i)} \quad  If $c_v < 0$ (i.e. $u$ and $v$ are in the first
column of $\Lambda^\circ$) then $v = u+1$ and $F_{\Lambda^\circ}(z_1, \dots ,
z_n)$ can be written as $F_{\Lambda^\circ}(z_1, \dots, z_n) = F_{u}(w_u, w_v) \cdot F$. \\
Starting with $F_{\Lambda^\circ}(z_1, \dots , z_n)$ written in the
ordering described in Proposition \ref{q-jimtheorem1} and simply
moving the term $F_{v-u}(w_u, w_v)$ to the left
results in all the singularity terms in the product $F$ remaining
tied to the same triple terms as originally described in that
ordering, and the index of $F_{v-u}(w_u, w_v)$ increases from $u-v$ to $u$. Therefore we may still form regular triples for
each singularity in $F$, and hence $F$ is regular at $z_1 = z_2 = \dots = z_n \neq 0$.\\
So by considering this expression for $F_{\Lambda^\circ}(z_1, \dots ,
z_n)$ at $z_1 = z_2 = \dots = z_n \neq 0$ we see that $F_{\Lambda^\circ}$ will be
divisible on the left by $F_{u}(q^{2c_u(\Lambda^\circ)}, q^{2c_v(\Lambda^\circ)}) = T_{v-u} - q$.

\emph{(ii)} \quad  If $c_v = 0$ then $v=s+1$, and $F_{\Lambda^\circ}(z_1,
\dots , z_n)$ can be written as \[ F_{\Lambda^\circ}(z_1, \dots , z_n) =
\prod_{i = u, \dots , s}^\leftarrow F_{i}(w_i, w_{s+1}) \cdot F' . \] Again, starting with the ordering
described in Proposition \ref{q-jimtheorem1}, this results in  all
the singularity terms in the product $F'$ remaining tied to the
same triple terms as originally described in that ordering. Hence
$F'$ is regular at $z_1 = z_2 = \dots = z_n \neq 0$. And so $F_{\Lambda^\circ}$
is divisible on the left by
\[ \prod_{i = u, \dots , s}^\leftarrow F_{i}(q^{2c_i(\Lambda^\circ)}, q^{2c_{s+1}(\Lambda^\circ)}).
\]

\emph{(iii)} \quad  If $c_v > 0$ (i.e. $v$ is above the steps)
then $F_{v-u}(w_u, w_v)$ is tied to the singularity
$F_{v-u+1}(w_{u-1}, w_v)$ as a triple term. To
show divisibility by $F_{v-u}(w_u, w_v)$ in this case
we need an alternative expression for $F_{\Lambda^\circ}(z_1, \dots ,
z_n)$ that is regular when $z_1 = z_2 = \dots = z_n \neq 0$. Define a
permutation $\tau$ as follows,

\[ \tau = \prod_{i = u, \dots, s}^\rightarrow \left( \prod_{j= s+1, \dots,
v}^\leftarrow (i j) \right)
\phantom{XXXXXXXXXXXXXXXXXXXXXXXXXXXXXXX}
\]
\[ = \left(
\normalsize \begin{array}{ccccccccccccccccc}
                   1 & 2 & \dots & u-1 & u & u+1 & \dots &  & \dots & \dots  &  & \dots & v-1 & v & v+1 & \dots & n \\
                   1 & 2 & \dots & u-1 & s+1 & s+2 & \dots & v-1 & v & u & u+1 & \dots & s-1 & s & v+1 & \dots & n \\
                 \end{array} \right) \]

From the definition of $C_1$ in (\ref{q-cproduct}) we now define
$C'_1 = \psi_\tau C_1$, where $\psi_\tau$ is a homomorphism such that
\[ \psi_\tau F_{j-i}(w_i,w_j) = F_{j-i}(w_i,w_{\tau j}). \]

Define $R'_1$ as,
\begin{eqnarray*}
R'_1 & = & \prod_{i= r+2, \dots, u-1}^\leftarrow \left(
\prod_{j=s+1, \dots ,v}^\rightarrow F_{i+j-s-r-1}(w_i,w_j) \right) \cdot \prod_{i=
s+1, \dots, t-1}^\rightarrow \left( \prod_{j=i+1, \dots
,t}^\rightarrow f_{j-i+1}(w_i,w_j) \right) \\
 && \times  \left( \prod_{j=s+1, \dots
,t}^\leftarrow F_{t+1-j}(w_{r+1},w_j) \right) \left( \prod_{j=t+1, \dots
,v}^\rightarrow F_{j-s}(w_{r+1},w_j) \right) \\
   && \times \prod_{i= r+1,
\dots, s-1}^\rightarrow \left( \prod_{j=i+1, \dots ,s}^\rightarrow
F_{j-i+v-s}(w_i,w_j) \right) \cdot \prod_{j= v+1, \dots, n}^\rightarrow
\left( \prod_{i=r+1, \dots ,s}^\rightarrow F_{j-i}(w_i,w_j) \right). \\
\end{eqnarray*}

Finally, define $C'_2$ as, \[ C'_2 = \prod_{j= t+1, \dots,
n}^\rightarrow \left( \prod_{i=s+1, \dots ,t}^\rightarrow F_{j-i}(w_i,w_j)
\right). \]

Then, \[ F_\Lambda(z_1, \dots , z_n) = \prod_{i = u, \dots ,
s}^\leftarrow \left( \prod_{j=s+1, \dots, v}^\rightarrow F_{i+j-s-1}(w_i,w_j)
\right) \cdot C'_1 R'_1 C'_2 R_2 \cdot \prod_{i=3}^d C_iR_i, \]
where $d$ is the number of principal hooks of $\lambda$.

The product $C'_1 R'_1 C'_2 R_2 \cdot \prod C_iR_i$ is regular at
$z_1 = z_2 = \dots z_n \neq 0$ since, as before, for any singularity
$(a,b)$ the terms $F_{i}(w_a,w_b)F_{i-1}(w_{a+1},w_b)$ can be replaced by the triple
$P_{i-1}^\pm F_{i}(w_a,w_b)F_{i-1}(w_{a+1},w_b)$ for some index $i$ -- except in the
expression $R'_1$ where the terms $F_{i-1}(w_a,w_l)F_{i}(w_a,w_b)$ are replaced by
$F_{i-1}(w_a,w_l)F_{i}(w_a,w_b)P_{i-1}^+$, where $l$ is the entry to the
immediate left of $b$. Note that $l = b-1$ when $c_b(\Lambda^\circ) > 1$ and $l =
s+1$ when $c_b(\Lambda^\circ) = 1$. \\
And so by letting $z_1 = z_2 = \dots =z_n \neq 0$ we see that $F_{\Lambda^\circ}$
is divisible on the left by \[ \prod_{i = u, \dots , s}^\leftarrow
\left( \prod_{j=s+1, \dots, v}^\rightarrow F_{i+j-s-1}(q^{2c_i(\Lambda^\circ)}, q^{2c_j(\Lambda)})
\right).
\]
\end{proof}

\begin{proposition}\label{q-jimtheorem3} Suppose the numbers $u < v$ stand next to each
other in the same row of the hook tableau $\Lambda^\circ$ of shape
$\lambda$. Let $r$ be the last entry in the column containing $u$.
If $c_u > 0$ then the element $F_{\Lambda^\circ} \in H_n$ is
divisible on the left by $F_{u}(q^{2c_u(\Lambda^\circ)}, q^{c_v(\Lambda^\circ)}) = T_{u} +
q^{-1}$. If $c_u \leqslant 0$ then the element $F_{\Lambda^\circ} \in
H_n$ is divisible on the left by the product
\[ \prod_{i = u, \dots, r}^\leftarrow \left( \prod_{j= r+1, \dots,
v}^\rightarrow F_{i+j-r-1}(q^{c_i(\Lambda^\circ)}, q^{c_j(\Lambda^\circ)}) \right) \]
\end{proposition}

We omit the proof of this proposition as it is very similar to
that of Proposition \ref{q-jimtheorem2}.

\begin{lemma}\label{q-divisibilitybyadjacenttransposition} Let $\Lambda$ and $\tilde{\Lambda}$ be tableaux of the same shape such that $k= \Lambda(a,b) = \Lambda(a+1,b)-1$ and $\tilde{k}=\tilde{\Lambda}(a,b) = \tilde{\Lambda}(a+1,b)-1$. Then $F_\Lambda \in H_n$ is divisible on the left by $T_k - q$ if and only if $F_{\tilde{\Lambda}} \in H_n$ is divisible on the left by $T_{\tilde{k}} - q$. \end{lemma}

\begin{proof}
Let $\sigma$ be the permutation such that $\tilde{\Lambda}=\sigma\cdot\Lambda$.
There is a decomposition
$\sigma=\sigma_{i_N}\ldots\sigma_{i_1}$ such that for each $M=1, \ldots , N-1$
the tableau $\Lambda_{ M}=\sigma_{i_M}\ldots\sigma_{i_1}\cdot\Lambda$ is
standard. Note that this decomposition is not necessarily reduced.

Denote $F_\Lambda$ by $F_k(q^{2c_k(\Lambda)},q^{2c_{k+1}(\Lambda)}) \cdot F$. Then, by using the relations
(\ref{q-triple}) and (\ref{q-commute}) and  Proposition \ref{q-jimtheorem1}, we have the following chain of equalities:

\[
F_{\tilde{k}}
\bigl( 
q^{2c_{\tilde{k}}(\tilde{\Lambda})}, q^{2c_{\tilde{k}+1}(\tilde{\Lambda})}
\bigr)
\ \cdot\ 
\prod_{M=1, \ldots , N}^{\longleftarrow}
F_{ n- i_M}
\bigl( 
q^{2 c_{i_M+1}(\Lambda_M)}
, 
q^{2 c_{i_M}(\Lambda_M)}
\bigr) \cdot F =  \]
\[\prod_{M=1, \ldots , N}^{\longleftarrow}
F_{ i_M}
\bigl( 
q^{2 c_{i_M}(\Lambda_M)}
, 
q^{2 c_{i_M+1}(\Lambda_M)}
\bigr)
\ \cdot\ 
F_{ k}
\bigl( q^{2c_k(\Lambda)}, q^{2c_{k+1}(\Lambda)}\bigr) \cdot F =\]
\[
\prod_{M=1, \ldots , N}^{\longleftarrow}
F_{ i_M}
\bigl( 
q^{2 c_{i_M}(\Lambda_M)}
, 
q^{2 c_{i_M+1}(\Lambda_M)}
\bigr)
\ \cdot\ F_\Lambda\ =
\]
\[
F_{\tilde{\Lambda}}\ \cdot
\prod_{M=1, \ldots , N}^{\longleftarrow}
F_{ n- i_M}
\bigl( 
q^{2 c_{i_M+1}(\Lambda_M)}
, 
q^{2 c_{i_M}(\Lambda_M)}
\bigr)
\]

Hence divisibility by $T_k - q$ for $F_\Lambda$ implies its counterpart for the tableau
$\tilde{\Lambda}$ and the index $\tilde{k}$, and vice versa.
Here we also use the equalities
\[
F_{ k}
\bigl( q^{2c_k(\Lambda)}, q^{2c_{k+1}(\Lambda)}\bigr)
=T_k-q,
\]\[
F_{\tilde{k}}
\bigl( 
q^{2c_{\tilde{k}}(\tilde{\Lambda})}, q^{2c_{\tilde{k}+1}(\tilde{\Lambda})}
\bigr)
=T_{\tilde{k}}-q.
\]
\end{proof}

\begin{corollary}\label{q-divisibilitycorollary}
If $k=\Lambda(a, b)$ and $k+1=\Lambda(a+1, b)$
then the element $F_\Lambda\in H_n$ is divisible on the left by $T_k-q$. If $k=\Lambda(a, b)$ and $k+1=\Lambda(a, b+1)$
then the element $F_\Lambda\in H_n$ is divisible on the left by 
$T_k+q^{-1}$.
\end{corollary}

\begin{proof}
Due to Lemma \ref{q-divisibilitybyadjacenttransposition}
it suffices to prove the first part of Corollary \ref{q-divisibilitycorollary} for only one 
standard tableau $\Lambda$ of shape $\lambda$. Therefore, using Proposition \ref{q-jimtheorem2} and taking $\tilde{\Lambda}$ to be the hook tableau $\Lambda^\circ$ of shape $\lambda$ we have shown the first part of Corollary \ref{q-divisibilitycorollary} in the case $c_v(\Lambda) < 0$.

Next let $\Lambda^{\circ}(a,b)= u$, $\Lambda^{\circ}(a+1,b) = v$ and $s$ be the last entry in the row containing $u$. Then for $c_v({\Lambda^\circ}) \geqslant 0$ Proposition \ref{q-jimtheorem2} showed that $F_{\Lambda^\circ} \in H_n$ is divisible on the left by 
\begin{equation}\label{q-divisibilitycorollaryequation2} \prod_{i = u, \dots, s}^\leftarrow \left( \prod_{j= s+1, \dots,
v}^\rightarrow F_{i+j-s-1}(q^{2c_i(\Lambda^\circ)}, q^{2c_j(\Lambda^\circ)}) \right) \end{equation}
Put $k=u+v-s-1$,
this is the value of the index $i+j-s-1$ in (\ref{q-divisibilitycorollaryequation2})
when $i=u$ and $j=v$. Let $\Lambda$ be the tableau
such that $\Lambda^\circ$ is obtained from 
the tableau $\sigma_k\cdot\Lambda$ by the permutation
\[
\prod_{i =  u, \ldots, s}^{\longleftarrow}\,
\biggl(\ 
\prod_{j = s+1, \ldots,  v}^{\longrightarrow}
\sigma_{ i+j-s-1} \biggr).
\]
The tableau $\Lambda$ is standard. Moreover, then
$\Lambda(a, b)=k$ and $\Lambda(a+1, b)=k+1$.
Note that the rightmost factor in the product (\ref{q-divisibilitycorollaryequation2}),
corresponding to $i=u$ and $j=v$, is
\[
F_{u+v-s-1}
\bigl( q^{2c_u(\Lambda^\circ)}, q^{2c_{v}(\Lambda^\circ)}\bigr)
=
T_k - q.
\]
Denote by $F$ the product of all factors in (\ref{q-divisibilitycorollaryequation2})
but the rightmost one. Further, denote by $G$ the product obtained
by replacing each factor in $F$
\[
F_{i+j-s-1}
\bigl(q^{2c_i(\Lambda^\circ)}, q^{2c_j(\Lambda^\circ)}\bigr)
\]
respectively by
\[
F_{n-i-j+s+1}
\bigl(q^{2c_j(\Lambda^\circ)}, q^{2c_i(\Lambda^\circ)}\bigr).
\]
The element $F \in H_n$ is invertible, and we have
\[ F \cdot F_\Lambda = F_{\Lambda^\circ} \cdot G = F \cdot (T_k - q) \cdot (C'_1R'_1C'_2R_2 \prod C_iR_i) \cdot G,\] where the final equality is as described in Proposition \ref{q-jimtheorem2}. Therefore the divisibility of 
the element $F_{\Lambda^\circ}$ on the left by the product (\ref{q-divisibilitycorollaryequation2})
will imply the divisibility of the element $F_\Lambda$ on the left by 
$T_k - q$.

This shows the required divisibility for the tableau $\Lambda = \sigma \sigma_k \cdot \Lambda^{\circ}$. Using Lemma \ref{q-divisibilitybyadjacenttransposition} again concludes the proof of the first part of Corollary \ref{q-divisibilitycorollary}. 

The second part of Corollary \ref{q-divisibilitycorollary} may be shown similarly.
\end{proof}


\section{Generating irreducible representations of $H_n$}

For every standard tableau $\Lambda$ of shape $\lambda$ 
we have defined an element $F_\Lambda$ of the algebra $H_n$. 
Let us now assign to $\Lambda$ another element of $H_n$,
which will be denoted by $G_\Lambda$. 

Let $\rho\in S_n$ be the permutation such 
that $\Lambda=\rho\cdot\Lambda^\circ$, that is
$\Lambda(a, b)=\rho(\Lambda^\circ(a, b))$ for all possible $a$ and $b$. For any $j=1, \ldots , n$ denote by $\mathcal{B}_j$ the subsequence of the sequence
$\rho(1), \ldots ,\rho(n)$ consisting of all $i<j$ such that
$\rho^{-1}(i)<\rho^{-1}(j)$. Let $|\mathcal{B}_j|$ be the length of sequence $\mathcal{B}_j$. 

Consider the rational
function taking values in $H_n$, of the variables $z_1, \ldots , z_n$
\[
\prod_{j=1,\ldots, n}^{\longrightarrow}
\biggl(\ 
\prod_{k=1,\ldots,|\mathcal{B}_j|}^{\longrightarrow}
\ F_{j-k}
\bigl( q^{2c_i(\Lambda)}z_i, q^{2c_j(\Lambda)}z_j\bigr)
\biggr)
\ \quad\textrm{where}\ \quad 
i=\mathcal{B}_j(k).
\]
Denote this rational function by $G_\Lambda(z_1, \ldots , z_n)$.

For any $j=1, \ldots , n$ take the subsequence of the sequence
$\rho(1), \ldots ,\rho(n)$ consisting of all $i<j$ such that
$\rho^{-1}(i)>\rho^{-1}(j)$. Denote by $\mathcal{A}_j$ the result of reversing
this subsequence. Using induction on the length of the element $\rho\in S_n$, one can prove that
\[
F_\Lambda(z_1, \ldots , z_n)\ =\  
G_\Lambda(z_1, \ldots , z_n)\ \ \times
\]\[
\prod_{j=1,\ldots, n}^{\longleftarrow}
\biggl(\ 
\prod_{k=1,\ldots,|\mathcal{A}_j|}^{\longleftarrow}
\ F_{ n-j+k}
\bigl( q^{2c_i(\Lambda)}z_i, q^{2c_j(\Lambda)}z_j\bigr)
\biggr)
\ \quad\textrm{where}\ \quad 
i=\mathcal{A}_j(k).
\]
Hence restriction of $G_\Lambda(z_1, \ldots , z_n)$ 
to the subspace $\mathcal{H}_\Lambda\subset\mathbb{C}(q)^{n}$
is regular on the line $(z_1 = \cdots = z_n \neq 0)$ due to Proposition \ref{q-jimtheorem1}.
The value of that restriction is our element $G_\Lambda\in H_n$ 
by definition. 

Note that $G_{\Lambda^\circ}=F_{\Lambda^\circ}$. Denote by $V_\lambda$ the left ideal in the algebra $H_n$
generated by the element $F_{\Lambda^\circ}$. The elements $G_\Lambda\in H_n$ for all
pairwise distinct standard tableaux $\Lambda$ of shape $\lambda$
form a basis in the vector space $V_\lambda$, \cite{N1}. Let us now consider the left ideal $V_\lambda\subset H_n$ as 
$H_n$-module. Here the algebra $H_n$ acts via left multiplication.

The following theorems are stated without proofs. All proofs can in found in \cite[Section 3]{N1}.

\begin{theorem}\label{C3.6}
{\bf\hskip-6pt.\hskip1pt} 
The $H_n$-modules $V_\lambda$ for different partitions $\lambda$ of $n$
are irreducible and pairwise non-equivalent.
\end{theorem}

Furthermore we have the following proposition about the elements $G_\Lambda$;

\begin{proposition}\label{Vkappa}
The vector $G_\Lambda\in V_\lambda$ belongs to the $H_k$-invariant subspace
in $V_\lambda$, equivalent to the $H_k$-module $V_{\kappa}$ where
the partition $\kappa$ is the shape of the tableau obtained by
removing from $\Lambda$ the entries $k+1, \ldots , n$.
\end{proposition}

The properties of the vector $G_\Lambda$ given 
by Proposition \ref{Vkappa} for $k=1, \ldots , n-1$,
determine this vector 
in $V_\lambda$ uniquely up to a non-zero factor from $\mathbb{C}(q)$. 
These properties can be restated for any irreducible
$H_n$-module $V$ equivalent to $V_\lambda$.
Explicit formulas for the action of the generators $T_1, \ldots , T_{n-1}$
of $H_n$ on the vectors in $V$ determined by these properties,
are known; cf.\ \cite[Theorem 6.4]{M}.

Setting $q=1$, the algebra $H_n$ specializes to the symmetric group ring
$\mathbb{C}S_n$. The element $T_\sigma\in H_n$ then specializes to the
permutation $\sigma\in S_n$ itself. 
The proof of Proposition \ref{q-jimtheorem1} demonstrates that the coefficients
in the expansion of the element $F_\Lambda\in H_n$ relative to the basis
of the elements $T_\sigma$, are regular at $q=1$ as rational functions
of the parameter $q$. Thus the specialization of the element
$F_\Lambda\in H_n$ at $q=1$ is well defined. The same is true for the
element $G_\Lambda\in H_n$.
The specializations at $q=1$ of the basis vectors $G_\Lambda\in V_\lambda$
form the \emph{Young seminormal basis} in the corresponding
irreducible representation of the group $S_n$.
The action of the generators $\sigma_1, \ldots ,\sigma_{n-1}$ of $S_n$ on the 
vectors of the latter basis was first given by \cite[Theorem IV]{Y2}.
For the interpretation of the elements $F_\Lambda$ and $G_\Lambda$ using
representation theory of the affine Hecke algebra $\widehat{H}_n$, 
see \cite[Section 3]{C2} and references therein.

\end{document}